\documentclass[11pt,a4paper]{scrartcl}

\PassOptionsToPackage{mathfont=stix2}{mystyle-koma-enhanced}
\usepackage{mystyle-koma-enhanced}
\usepackage{comment}
\usepackage{tikz}
\usetikzlibrary{backgrounds,patterns,positioning,shapes.geometric}
\usepackage{bbm}

\MSC{60K35, 82B43, 05C81}
\Keywords{Poisson zoo, loop soup, percolation, metric graph}

\renewcommand{\P}{\mathbb{P}}
\newcommand{\E}{\mathbb{E}}
\newcommand{\N}{\mathbb{N}}

\newcommand{\0}{\mathbf{o}}

\title{(Non-)coincidence of critical parameters for Poisson Zoos and Loop Soup Percolation\\ on \texorpdfstring{$\mathbb Z^d$}{Z\^d}, $d > 4$, and \texorpdfstring{$\mathbb T_d$}{T\_d}, $d \ge 3$}
\author{
Volker Betz \orcidlink{} \\ \texttt{\small betz@mathematik.tu-darmstadt.de} \\
\and
Alexander Drewitz \orcidlink{0000-0002-5546-3614} \\ \texttt{\small adrewitz@uni-koeln.de} \\
\and
Andreas Klippel \orcidlink{} \\ \texttt{\small anklippe@uni-mainz.de} \\
\and
Christian M\"{o}nch \orcidlink{0000-0002-6531-6482} \\ \texttt{\small cmoench25@gmail.com}
}
\date{\today}

\begin{document}
\maketitle

\begin{abstract}
    In this article we investigate the (non)-coincidence of critical parameters for various related percolation problems. More precisely, for the random walk loop soup we show that on $\mathbb Z^d$, $d\ge 5$, the critical parameters for the percolation problems differ on the discrete graph and the respective metric graph. Moreover, on trees we deduce an analogous statement as well as the coincidence of the critical parameters for percolation and susceptibility for a more general class of percolation problems, the so-called Poisson zoo.
    Along the way we develop the useful notion of sensitivity to Bernoulli enhancements of such percolation problems with long range correlations, which builds on previously developed enhancement ideas.
\end{abstract}

\section{Introduction}\label{sec:intro}

Models of percolation with long-range correlations have been a very active
area of research in the last decade, with notable progress e.g.\ for random-cluster
and Potts models \cite{DuminilCopinRaoufiTassion2019}, loop soup percolation
and loop clusters \cite{LeJanLemaire13,ChangSapo16,Lupu2016}, including
Brownian loop soup clusters in Euclidean spaces
\cite{LawlerWerner2004,We-20,jego2026threedimensionalbrownianloopsoup}, the closely related
Gaussian free field excursion sets \cite{BrLeMa-87,RoSz-12,DPR18}, and spatial random
permutations \cite{BetzUeltschi2009}. In this article we consider two closely
related loop models: the (discrete) random walk loop soup and its metric graph, or
cable system, counterpart. Part of our results will be formulated in the
more general setup of so-called Poisson zoos \cite{pete2025nonamenablepoissonzoo}.

A main focus in the investigation of such models is the understanding of
percolation phase transitions, i.e.\ the emergence of infinite connected
components when one sweeps over the parameter range. Broadly speaking, these
problems are still significantly less understood than in the classical setting of
Bernoulli percolation. A natural question in loop soup percolation is whether passing from the discrete graph to its metric graph, or cable system, counterpart changes the percolation threshold. The cable system retains the discrete fundamental loops but supplies additional  connectivity through smaller loops. 
We prove that this extra connectivity strictly lowers the critical parameter for the random-walk loop soup on  $\mathbb Z^d,$ $d\geq5$, and on every regular tree of degree at least three. (It should be noted that such threshold separation of parameters is not automatic: for the loop soup in the killed half-plane, the discrete and cable system thresholds both equal 1/2.)
More generally, we develop a notion of sensitivity for Poisson-zoo percolation and show that, under a finite-overlap assumption on regular trees, every positive independent Bernoulli enhancement yields percolation at some subcritical intensity already.

More precisely, on a connected graph $G = (V,E)$ with
at most countable vertex set and bounded degree,
the underlying process is the continuous-time nearest-neighbor random walk on
the vertex set $V$; at each jump time it crosses one of the incident edges,
chosen uniformly \cite{LawlerWerner2004,LeJanStFlour}. A loop is obtained by conditioning this process to return to
its starting vertex after a positive time and then recording its vertex trace.
Equivalently, after forgetting holding times, the non-trivial trace is a finite
cyclic nearest-neighbor path. In particular, loops are ``discrete'' graph objects in contrast to the metric graph loops discussed below.
The set of all loops is thus countable, and to  each loop $\ell$, we can assign
a probability $p(\ell)$ with which the loop is present. Each loop is independently either
present or absent. The values of $p(\ell)$ are
 specific, and are given in the next section, but for now it is enough to know
that they decay to zero when the length $|\ell|$ of $\ell$ increases to infinity; furthermore, they
depend on an intensity parameter $\alpha$ with $\lim_{\alpha \to 0} p_\alpha(\ell) = 0$.
This induces a percolation model by declaring
an edge $e \in E$ to be open if a loop $\gamma$ is present that crosses this edge.
Bernoulli edge percolation with probability $p$ is a special case of this construction
when we set $p(\ell):=p$ for each unrooted loop with $|\ell|=2$, and
$p(\ell):=0$ otherwise.
The study of loop clusters and their percolative properties has emerged naturally \cite{LeJanLemaire13,ChangSapo16,We-20,MR5067664}.

The cable system, or metric graph, in turn is obtained by replacing each edge
$e=\{x,y\}$ by a line segment of length $1/2$ whose endpoints are identified with $x$ and $y$, respectively.
On this locally one-dimensional space one can run Brownian motion, with the convention
that whenever the path reaches a vertex it chooses one of the incident cables
uniformly for its next excursion (see \cite{Lupu2016} for a definition). The metric graph loop soup is the
corresponding Poisson ensemble of Brownian loops (with intensity measure given in \eqref{eq:metricIntensity}).

A metric graph loop may avoid the vertices of the underlying discrete graph altogether, or may hit a single vertex without
making a non-trivial graph loop, in the sense of its trace on the discrete graph consisting of at least two vertices. The latter non-trivial traces are the \emph{fundamental} loops of the metric graph loop soup, and their union on $V$ has the same law as that of the discrete loop soup (with the same intensity).

The metric graph also induces its own percolation problem: two points are
connected if one can travel from one to the other by visiting finitely many loops each of which intersects the trace of their predecessor. An edge may be connected
through cable portions covered by loops that do not produce a discrete crossing
of that edge, or by excursions attached to a fundamental loop but not recorded
in its vertex trace. As a consequence, the cable system can be viewed as the discrete loop
soup together with an additional (dependent) enhancement.

This brings the models into the realm of enhanced percolation. A base
percolation model is first sampled at intensity $\alpha$, and then an
additional, independent or dependent, mechanism is added with strength
$\delta>0$. One asks whether every positive enhancement can create
percolation while the base model is still subcritical. If this is possible,
we say that the base model is \emph{sensitive to enhancement}.
As mentioned above, such sensitivity is not as obvious as it may seem, with the killed half-plane providing a counterexample, recalled in Remark~\ref{rem:loop soup-threshold-discussion} \ref{item:half-plane-no-sensitivity}.

This additional cable system connectivity is (in spirit) also reflected in recent work on
critical cable graph loop soups and Gaussian free fields. 
Similarly to the percolation problems of the discrete loop soup, the percolative study of the metric graph loop soup and its associated Gaussian free
field have experienced tremendous research activity, see \cite{MR4914069,DiWi,MR4421175,cai2026gapclusterdimensionsloop} and references therein for an incomplete survey.

Section \ref{sec:MR} introduces Poisson-zoo percolation and proves the tree results on the coincidence of critical parameters (Theorems~\ref{thm:zsharp}) and enhancement sensitivity (Theorem \ref{thm:zenhan}). Section \ref{sec:LS} specializes the framework to discrete loop soups and derives threshold separation on regular trees; it then also proves the corresponding result on  $\mathbb Z^d$, $d\ge5$, by a separate one-generation comparison. Section \ref{sec:trees} contains the proofs of the Poisson-zoo theorems.

\section{Models and results}\label{sec:MR}

Let $G=(V,E)$ be a vertex-transitive graph with countable vertex set $V$ and
finite degree. Define
\[
\mathfrak A:=\{A\subseteq G:\ A\text{ is a finite connected subgraph}\},
\]
and for $A\in\mathfrak A$, write $V(A)$ and $E(A)$ for its vertex and edge sets.
The elements of $\mathfrak A$ will be called \emph{graph animals}. Let
$\mu:\mathfrak A\to[0,\infty)$ be invariant under graph automorphisms, meaning
that $\mu(A)=\mu(\phi(A))$ for every $A\in\mathfrak A$ and every graph
automorphism $\phi$ of $G$. For $\alpha>0$, a \emph{Poisson zoo} (a name coined in
\cite{RathRokob22,pete2025nonamenablepoissonzoo}) of intensity $\alpha\mu$ is a
family $(N_A^\alpha)_{A\in\mathfrak A}$ of independent random variables, where
$N_A^\alpha$ is Poisson-distributed with parameter $\alpha\mu(A)$. We define the
occupied edge set of the zoo as
\[
\mathcal O^\mu_\alpha
:=
\bigcup_{A:\,N_A^\alpha\ge1}E(A).
\]
For $\0\in V$, we write $C^\mu_\alpha(\0)$ for the connected component of the
graph induced by $\mathcal O^\mu_\alpha$ that contains $\0$.
Classical Bernoulli bond percolation with parameter
$p=1-\exp(-\alpha)$ is the special case obtained by putting $\mu(A)=1$ whenever
$|E(A)|=1$, and $\mu(A)=0$ otherwise; more generally, however, this model covers a broad
 range of percolation models with long range correlations, and each $A\in\mathfrak A$ is present independently with probability
$1-\exp(-\alpha\mu(A))$.
As a further example, one obtains a {\em discrete Boolean model} by choosing centers
according to a Poisson process on $V$, attaching to each center an independent
finite random radius, and occupying the ball around the center; in the zoo
language, the animals are precisely these finite balls, with $\mu$ given by the
intensity of the center-radius pair.

The example most relevant to us is induced by the (trace of the) random walk loop soup. Let
$(X_t)_{t\geq0}$ be the continuous-time variable-speed nearest-neighbor
(simple) random walk on $G$, with heat kernel $p_t(x,y)$.
For \(t>0\), let \(D([0,t],V)\) denote the space of right-continuous
paths with left limits from \([0,t]\) to \(V\), where \(V\) carries the
discrete topology. Define the space of based nearest-neighbor loops by
\[
\begin{aligned}
\mathcal L^{\mathrm b}
:=
\bigl\{(t,\gamma):\;&t>0,\ \gamma\in D([0,t],V),\
\gamma(0)=\gamma(t),\\
&\gamma\text{ has finitely many jumps, each traversing an edge of }G
\bigr\}.
\end{aligned}
\]
For \((t,\gamma)\in\mathcal L^{\mathrm b}\), extend \(\gamma\)
periodically to \(\mathbb R\) and define
\[
(\theta_s\gamma)(r):=\gamma(r+s),
\qquad s\in[0,t),\quad r\in[0,t].
\]
We identify \((t,\gamma)\) and \((t',\gamma')\) if \(t=t'\) and
\(\gamma'=\theta_s\gamma\) for some \(s\in[0,t)\). The resulting quotient
\[
\mathcal L:=\mathcal L^{\mathrm b}/\!\sim
\]
equipped with its quotient sigma-field is the space of unrooted oriented
loops.

For $\gamma\in\mathcal L$, write
$\operatorname{tr}(\gamma)$ for its graph trace: the finite connected subgraph
whose vertices are visited by $\gamma$ and whose edges are crossed by it. Set
\[
\mathcal L^{\mathrm{nt}}
:=
\bigl\{\gamma\in\mathcal L:
E\bigl(\operatorname{tr}(\gamma)\bigr)\neq\varnothing
\bigr\}.
\]
Thus $\mathcal L^{\mathrm{nt}}$ consists of the non-trivial loops, namely those
that make a non-zero number of jumps. For $x\in V$ and $t>0$, let
$\mathbb P^t_{x,x}$ be the random-walk bridge probability measure of duration
$t$, starting and ending at $x$. We define the non-trivial Markov loop measure
$\nu$ on $\mathcal L$ by
\begin{equation}\label{eq:discreteLoopMeasure}
\nu(B)
:=
\sum_{x\in V}\int_0^\infty
\frac{1}{t}\,
\mathbb P^t_{x,x}\bigl(B\cap\mathcal L^{\mathrm{nt}}\bigr)
p_t(x,x)\,{\mathrm d}t,
\qquad B\subseteq\mathcal L.
\end{equation}
The restriction to $\mathcal L^{\mathrm{nt}}$ discards zero-jump loops before
the trace map is applied. Such loops occupy no edge and hence do not affect any
percolation event, but retaining them would give infinite mass to singleton
traces.

The graph-animal intensity is the push-forward of $\nu$ under the trace map:
\begin{equation}\label{eq: loop soup intensity}
\mu:=\operatorname{tr}_{\#}\nu,
\qquad
\mu(A)
=
\nu\bigl(\{\gamma\in\mathcal L^{\mathrm{nt}}:
\operatorname{tr}(\gamma)=A\}\bigr),
\qquad A\in\mathfrak A.
\end{equation}
The Poisson point process on $\mathcal L^{\mathrm{nt}}$ with intensity
$\alpha\nu$ is the non-trivial Markovian loop soup. Its image under the trace
map is the Poisson zoo with intensity $\alpha\mu$, and the two processes induce
the same occupied edge set. We will use this identification below; see
\cite{LeJanStFlour,LeJanLemaire13,ChangSapo16}.

We write
\[
\chi_\mu(\alpha) := \mathbb E[|C^\mu_\alpha(\0)|]
\]
for the expected size of the cluster containing $\0$ (also referred to as {\em susceptibility}).
By the transitivity of $G$ and the automorphism invariance of $\mu$, it is independent of the specific choice of
$\0$. As usual in percolation, we introduce the percolation threshold
\begin{equation}\label{eq:alphac}
\alpha_{\mathsf c}^\mu := \alpha_{\mathsf c}^\mu(G) := \sup \big\{ \alpha \geq 0: \mathbb P(|C^\mu_\alpha(\0)| =
\infty) = 0\big\}
\end{equation}
and the susceptibility threshold
\begin{equation} \label{eq:alphaSharp}
\alpha_{\#}^\mu := \alpha_{\#}^\mu(G) :=\sup \{ \alpha \geq 0: \chi_\mu(\alpha) < \infty \}.
\end{equation}
The inequality $\alpha_{\mathsf c}^\mu \geq \alpha_{\#}^\mu$ is immediate.  Since replacing $\alpha$ by $\alpha + \delta$ just means putting an independent
Poisson zoo of intensity $\delta \mu$ on top of one of intensity $\alpha \mu$,
both $\chi_\mu(\alpha)$ and $\mathbb P(|C^\mu_\alpha(\0)| = \infty)$ are
non-decreasing
functions of $\alpha$.
A classical and important question in percolation models is whether
$\alpha_{\mathsf c}^\mu = \alpha_{\#}^\mu$.
Our first result establishes this equality on regular trees.
To state it,
we introduce the {\em overlap kernel}
\[
W_\mu(u,v) := \sum_{A \in \mathfrak A: u,v \in V(A)} \mu(A)
\]
which is a measure for how easy it is to connect $u$ and $v$ by a graph animal.
We will assume throughout that the overlap constant is finite:
\begin{equation} \label{eq: finite C_ast}
C_\ast(\mu) := \sup_{u \in V} \sum_{v \in V} W_\mu(u,v)  < \infty.
\end{equation}
Due to the vertex-transitivity of
$G$ and the invariance of $\mu$, it suffices to study a fixed
$u \in V$ instead of taking the supremum.
Note also that if for some $\theta > 0$ and $x \in V$ we have
\begin{equation}\label{eq:zexp}
\sum_{A\in\mathfrak A:\,x\in V(A)}
\mu(A)\mathrm{e}^{\theta |V(A)|}<\infty,
\end{equation}
then \eqref{eq: finite C_ast} holds: with
$C_\theta:=\sup_{n\ge1}n\mathrm e^{-\theta n}<\infty$, one has
$|V(A)|\le C_\theta\mathrm e^{\theta|V(A)|}$, and the identity below applies.

Condition
\eqref{eq: finite C_ast} has a natural interpretation.
First of all, we have
\[
C_\ast(\mu) = \sum_{v \in V}
\sum_{A \in \mathfrak A: \0,v \in A} \mu(A) =
\sum_{A \in \mathfrak A: \0 \in A} |V(A)| \mu(A).
\]
Letting $N_A$ denote the Poisson count of animal $A$ in the unit-intensity zoo, we have $\mathbb P(N_A^1 \geq 1) =
\mu(A) + O (\mu(A)^2)$ as $\mu(A) \to 0$, and condition \eqref{eq: finite C_ast}
is fulfilled if and only if
\[
\infty > \sum_{A \in \mathfrak A: \0 \in A} |V(A)|
\mathbb P(N_A^1 \geq 1) =
\mathbb E
\Big[ \sum_{A \in \mathcal A_{\0}^{1}} |V(A)| \Big],
\]
where
$\mathcal A_{\0}^{1}:=\{A\in\mathfrak A:\ \0\in V(A),\ N_A^1>0\}$.

The finite range case is straightforward. If there is an $r<\infty$ such
that $\mu(A)=0$ whenever $\operatorname{diam}(A)>r$, then only finitely many
animals with positive intensity can contain a fixed vertex, and
\eqref{eq: finite C_ast} follows immediately. This is the usual
$r$-dependent setting, where domination by product measures is a standard
comparison tool; see Liggett--Schonmann--Stacey \cite{LSS1997}.

Our main interest is the infinite range case, in particular the Markovian loop
soup, where traces can have arbitrarily large diameter. In this setting
\eqref{eq: finite C_ast} and \eqref{eq:zexp} are a
tail assumptions. We start with a result in the tree case. The strength of the conclusion depends on the strength of the tail assumption imposed on $\mu$.

\begin{theorem} \label{thm:zsharp}
Let $d\geq3$ and let
$\mu$ be an invariant intensity
function on the graph animals of $\mathbb T_d$.
\begin{enumerate}
\item\label{item:coincidence}

  If $\mu$ satisfies \eqref{eq: finite C_ast},
then
\[
\alpha_{\mathsf c}^\mu(\mathbb T_d)=\alpha_{\#}^\mu(\mathbb T_d).
\]

\item \label{item:exponentialCrit}
If $\mu$ satisfies \eqref{eq:zexp}, then for each $\alpha <  \alpha_{\mathsf c}^\mu(\mathbb T_d)$, there
exist constants $C_{\mu,\alpha},c_{\mu,\alpha}>0$ such that for $x \in V(\mathbb T_d)$,
\[
\P\big(|C^\mu_\alpha(x)|\ge n\big)
\le C_{\mu,\alpha}e^{-c_{\mu,\alpha}n}
\qquad\text{for all }n\ge1.
\]
\end{enumerate}
\end{theorem}

\begin{remark}
      For the important model of random walk loop soup, i.e.\ $\mu$ as in \eqref{eq: loop soup intensity}, the coincidence of the critical thresholds on regular trees (Theorem \ref{thm:zsharp} \ref{item:coincidence}) was recently established by \textcite{makowiec2026criticalcurvelooppercolation}, albeit with a technique that is different from our argument. Previously, on the Euclidean lattice, Chang and Sapozhnikov \cite[Theorem 1.7]{ChangSapo16} had shown that $\alpha_{\mathsf c}(\mathbb Z^d)$ and $\alpha_{\#}(\mathbb Z^d)$ exhibit the same first order behavior as $d \to \infty$. Our Theorem \ref{thm:zsharp} \ref{item:coincidence} is a further indication that (at least) in the mean field regime one might expect the coincidence of these critical parameters also for more general $\mu$.
\end{remark}

A concept closely related to the proof of Theorem~\ref{thm:zsharp}, and
interesting in its own right, is sensitivity with respect to Bernoulli
enhancement; see e.g.\ \cite{MR2446487} for an overview of previous enhancement methods, mostly applied to Bernoulli percolation. Given $\delta\in(0,1]$ and a Poisson zoo with intensity function $\mu$,
let $\mathcal O_\delta$ be an independent Bernoulli edge percolation with edge
probability $\delta$. Define
\[
C_{\alpha,\delta}^\mu(x)
:=
\text{the cluster of }x\text{ in }\mathcal O_\alpha^\mu\cup\mathcal O_\delta.
\]
We use this terminology only for models with a genuine, non-trivial
percolation transition, i.e.\ $0<\alpha_{\mathsf c}^\mu<\infty$.

For such a zoo, we say that it is \emph{$\delta$-sensitive with respect to
Bernoulli enhancement} if there exists $\alpha<\alpha_{\mathsf c}^\mu$ such
that
\begin{equation} \label{eq: enhancement}
\mathbb P \big(|C_{\alpha,\delta}^\mu(x)|=\infty\big)>0.
\end{equation}
In words, opening independent Bernoulli edges with probability $\delta$ can
push the model across its percolation threshold. The zoo is
\emph{infinitesimally sensitive with respect to Bernoulli enhancement} if it is
$\delta$-sensitive with respect to Bernoulli enhancement for every
$\delta\in(0,1]$.

\begin{remark}
There are two related, but conceptually different ways in which additional randomness is used to compare critical parameters.
\begin{enumerate}
    \item
The origin of using enhancing configurations is the theory of
\emph{essential enhancements} introduced by Aizenman and Grimmett
\cite{AizenmanGrimmett91}; see also \cite{MR2446487,BalisterBollobasRiordan14}.
In that setting one starts from, say, independent Bernoulli percolation and adds a
local monotone rule, driven by additional independent variables. The enhancement
is called essential if, at the deterministic level, activating the rule can turn
a configuration with no bi-infinite open path into one with such a path; the
main conclusion is then that the critical point is shifted strictly. The cable
system enhancement considered here has a similar monotone flavor, but it is not
literally an essential enhancement in this sense: the added connectivity comes
from a continuum Poisson ensemble and may be dependent and non-local when viewed
from the discrete graph. We therefore use {\em enhancement} in the broader sense
of an additional source of connectivity, and formulate the quantitative
strictness question below as {\em sensitivity to Bernoulli enhancement}.

\item
A different viewpoint is the direct coupling viewpoint from ordinary Bernoulli
percolation. Here one constructs two percolation configurations on the same
probability space, for example by obtaining the larger configuration from the
smaller one through an independent sprinkling of additional open edges, or by
obtaining the smaller one from the larger one through thinning. The comparison is
then made directly at the level of clusters or exploration processes: the added
edges may connect large finite clusters and create percolation, while deleting
edges may break an infinite cluster. This is the standard monotone coupling
picture underlying many classical comparisons of critical probabilities; see,
for instance, Kesten's book \cite{Kesten82}. Recent examples
include strict-inequality results for random-interchange-type models
\cite{Muhlbacher21,KLM25,betz2024loop,taggi2023essential}, strict monotonicity for independent long-range
percolation \cite{BM25}, and the coupling approach to stochastic
domination and graph fibrations in \cite{MPR25}. Our approach belongs to this second category.   \end{enumerate}
\end{remark}

As an easy example, if Bernoulli percolation on $G$ has critical parameter strictly
below one (see e.g.\ \cite{DuGoRaSeYa-20} for a recent overview of the current
state of the art), then the Poisson zoo corresponding to Bernoulli percolation
is infinitesimally sensitive in this sense. Also, any such zoo is
$\delta$-sensitive whenever $\delta\in(0,1]$ is already above the Bernoulli percolation
threshold. These simple examples make infinitesimal sensitivity look natural.
Note however that such sensitivity is not automatic: for
the killed half-plane loop soup discussed in
Remark~\ref{rem:loop soup-threshold-discussion}\ref{item:half-plane-no-sensitivity},
the cable system enhancement does not lower the critical point. 

Our second main result proves the infinitesimal sensitivity for Poisson zoos on the
$d$-regular tree, $d\ge3$, under the finite-overlap assumption.

\begin{theorem}
\label{thm:zenhan}
Let $d\geq3$ and set $G=\mathbb T_d$. Consider a Poisson zoo on $G$ with
intensity function $\mu$, and assume that \eqref{eq: finite C_ast} holds.
Assume moreover that the zoo is
non-degenerate, in the sense that $\mu(A)>0$ for some animal $A$ with
$E(A)\ne\varnothing$. Then the zoo is infinitesimally sensitive with respect to
Bernoulli enhancement.
\end{theorem}
The main reason this result is relevant in our setup is that it directly entails Theorem \ref{thm: loop soup} \ref{item:treeStrict} below (and in spirit also
is lurking behind the proof of Theorem \ref{thm:zsharp}
\ref{item:coincidence}).  The proof of Theorem \ref{thm:zenhan} will be postponed to Section \ref{sec:trees}.

\section{Discrete and metric graph loop soup} \label{sec:LS}

The cable system, or metric graph loop soup, was introduced in \cite{Lupu2016}. Intuitively, it corresponds to replacing the above continuous-time random walk loops by Brownian loops on the continuum graph obtained by gluing nearest neighbors of $G=(V,E)$ with line segments; we denote the resulting graph by $\widetilde G$ and its Brownian loop space by $\widetilde{\mathcal L}$.
It is interesting in its own right and serves as a comparison to the Markovian loop soup. To keep the loop-space measure distinct from the graph-animal intensity $\mu$, let $\widetilde\nu$ denote the metric-graph loop measure
\begin{equation} \label{eq:metricIntensity}
\widetilde\nu(B)
:=
\int_{\widetilde G}\int_0^\infty
\frac1t\,\widetilde{\mathbb P}^{\,t}_{x,x}(B)\,
\widetilde p_t(x,x)\,{\mathrm d}t\,\lambda({\mathrm d}x),
\qquad B\subseteq\widetilde{\mathcal L},
\end{equation}
where $\widetilde{\mathbb P}^{\,t}_{x,x}$ is the Brownian bridge probability measure on $\widetilde G$ of duration $t$, starting and ending at $x$, $\widetilde p_t$ is the heat kernel on $\widetilde G$, and $\lambda$ is its length measure, normalized so that each cable has length $1/2$. The metric-graph loop soup at level $\alpha$ is the Poisson point process with intensity $\alpha\widetilde\nu$. We refer to \cite{Lupu2016, MR4421175} and references therein for further detail.

There is a natural projection from Brownian loops on the metric graph to
(possibly empty) discrete loops, obtained by recording successive visits to the
vertices of the underlying graph (and adding the closing graph step when
needed). Conversely, one can recover the metric-graph loop soup from the
discrete soup plus additional randomness. We refer to \cite{Lupu2016} and
references therein for further details.

Recall that metric loops with non-empty discrete edge trace are called fundamental.
Given a cable system configuration of the loop soup, two graph vertices are connected if they can
be joined by a path in the metric graph covered by the union of finitely many Brownian
loops.

Some relevant facts about the cable system are the following:
\begin{itemize}
    \item Under the discrete projection, the point process of fundamental loops
    has intensity $\alpha\nu$ and hence is precisely the non-trivial Markovian
    loop soup with the same parameter $\alpha$.

    \item On $\mathbb T_d$ (and likewise on $\mathbb Z^d$), the cable system
    contains an independent Bernoulli enhancement of the discrete soup. Indeed,
    write $I_e$ for the cable corresponding to $e$ and choose
    $0<\eta<1/4<r<1/2$. For an oriented edge $(x,e)$, let
    $\mathscr H_{x,e}$ be the class of non-fundamental metric loops which hit
    $x$, have an excursion in $I_e$ reaching distance $r$ from $x$, and have no
    excursion of height $\eta$ in any other cable incident to $x$. These classes
    are pairwise disjoint, and the one-dimensional Brownian excursion
    decomposition at $x$ gives
    \[
    h_d:=\widetilde\nu(\mathscr H_{x,e})\in(0,\infty),
    \]
    independently of $(x,e)$. Hence their level-$\alpha$ counts are independent
    Poisson variables, also independent of the fundamental soup. The variables
    \[
    B_e:=\mathbbm 1\{N_{\mathscr H_{x,e}}\ge1,
                         N_{\mathscr H_{y,e}}\ge1\},
    \qquad e=\{x,y\},
    \]
    are therefore iid Bernoulli with parameter
    \[
    q_{\mathrm{cab}}(\alpha)=\bigl(1-e^{-\alpha h_d}\bigr)^2>0.
    \]
    On $\{B_e=1\}$, the two loop traces cover $I_e$ from its two endpoints
    beyond its midpoint, and hence cover the whole cable. Thus cable
    connectivity dominates the discrete soup enhanced by this independent
    Bernoulli field.
    
    \item For the cable system, the critical parameter for the percolation problem is
    \begin{equation}\label{eq:alphaTilde}
    \widetilde \alpha_{\mathsf c} := \widetilde \alpha_{\mathsf c}(G) := \sup \big \{ \alpha \geq 0: \mathbb P(|\widetilde C_\alpha(\0)| =
    \infty) = 0 \big\},
    \end{equation}
with $\widetilde C_\alpha(\0)$ denoting the connected component of the origin.

It turns out that
    $\widetilde \alpha_{\mathsf c}=1/2$ for a wide range of vertex-transitive graphs. This follows
    from \cite{Lupu2016,ChangDuLi24} and is related to the coupling with the Gaussian free field.
\end{itemize}

A very interesting and natural question (which is also an open question of Cai and Ding \cite{cai2026gapclusterdimensionsloop}, see end of Section 1 therein) is
whether
the
percolation threshold of the Markovian loop soup on $\mathbb Z^d$ in low dimensions is strictly
larger than the cable system threshold $1/2$. By the facts above, a sufficient
condition for such a gap is the sensitivity of the random walk loop
soup to Bernoulli enhancement of small enough parameter (which in particular is satisfied if it is infinitesimally sensitive). On $\mathbb T_d$, which should be thought of as a test lab for the general mean field regime, Theorem~\ref{thm:zenhan}
therefore gives different thresholds for the loop soup and the cable system
once we check \eqref{eq: finite C_ast} for the loop soup (this is done in the proof of Theorem \ref{thm: loop soup}
\ref{item:treeStrict} below). A limitation of the
proof of Theorem~\ref{thm:zenhan} is that it relies heavily on the recursive
structure of $\mathbb T_d$, so it is not immediate how to extend the argument to
$\mathbb Z^d$. Moreover, for $3\le d\le4$, it is known that for the random set
\begin{equation} \label{eq:oneStepSet}
C_\alpha^\mu(\0,1)
:=
\left\{
v\in \mathbb Z^d:\ \text{there exists }A\in\mathfrak A
\text{ with }N_A^\alpha\ge1,\ \0\in V(A),\ v\in V(A)
\right\}.
\end{equation}
of vertices reachable from $\0$ using a single occupied animal one has \(\mathbb E(|C_\alpha^\mu(\0,1)|) =\infty\) once $\alpha > 0$; see Lemma 5.1 of
\cite{ChangSapo16}. Referring back to the discussion below
\eqref{eq: finite C_ast}, write
$\mathcal A_{\0}^{\alpha}:=\{A\in\mathfrak A:\ \0\in V(A),\ N_A^\alpha\ge1\}$.
Then
\[
\mathbb E[|C_\alpha^\mu(\0,1)|] = \mathbb E
\Big[\Big| \bigcup_{A \in \mathcal A_{\0}^{\alpha}} V(A) \Big|\Big] \leq
\mathbb E \Big[ \sum_{A \in \mathcal A_{\0}^{\alpha}} |V(A)| \Big],
\]

Thus $\mathbb E[|C_\alpha^\mu(\0,1)|] =\infty $ implies
$C_\ast(\mu)=\infty$, so the above approach cannot apply to $\mathbb Z^d$ for
$d=3,4$. Nevertheless, we can show the strict inequality for the percolation thresholds for
the random walk loop soup and the cable system on $\mathbb Z^d$ for $d\ge5$ by
more direct means. We
summarize this discussion as follows.

\begin{theorem} \label{thm: loop soup}
\begin{enumerate}
\item \label{item:strict} For  $G=\mathbb Z^d$ and
$d\ge5$
we have the strict inequality of critical parameters when comparing the discrete graph with its metric graph version:
\begin{equation}\label{eq:strict}
\frac12 = \widetilde \alpha_{\mathsf c}(G) < \alpha_{\#}(G)\,  \big(\leq \alpha_{\mathsf c}(G)\big).
\end{equation}

\item \label{item:treeStrict} On $\mathbb T_d$, $d \geq 3$, the random walk loop soup
is
infinitesimally sensitive to Bernoulli enhancement; a fortiori, \eqref{eq:strict} holds true for $G=\mathbb T_d$ as well.
\end{enumerate}
\end{theorem}

\begin{remark}
\label{rem:loop soup-threshold-discussion}
\begin{enumerate}
\item Let us note here that heuristically the results of Theorems \ref{thm:zenhan} and  \ref{thm: loop soup} \ref{item:treeStrict} concerning $\mathbb T_d$ are not overly surprising. Indeed, on $\mathbb T_d$ the probability that a random walk loop visits a set of cardinality $n$ is roughly decaying exponentially in $n$; as a consequence, one has strongly decaying correlations which might be somewhat indicative of Theorems \ref{thm:zenhan} and  \ref{thm: loop soup} for $\mathbb T_d$.

The situation for $\mathbb Z^d$ is more intricate.
The probability of long random walk excursions only decays polynomially and our approach to establishing the strict inequality is based on an ad-hoc approach specifically tailored to the loop soup.

\item
\label{item:half-plane-no-sensitivity}
On the one hand, note that \cite[Thm.\ 1.7]{ChangSapo16} implies that $\alpha_{\mathsf c}(\mathbb Z^d) \to \infty$  as $d \to \infty$ and that $\widetilde \alpha_{\mathsf c}(\mathbb Z^d)=1/2$ for all $d \ge 3$ \cite{ChangDuLi24}. As a consequence, one can infer that $\widetilde \alpha_{\mathsf c}(\mathbb Z^d) < \alpha_c(\mathbb Z^d)$ for all $d$ large enough. The above Theorem \ref{thm: loop soup} \ref{item:strict} makes this quantitative in the dimension via establishing that this strict inequality holds from $d\ge 5$ onwards.

On the other hand, let $\mathbb H=\mathbb Z\times\mathbb N_*$, and consider the
loop soup associated with the simple random walk on $\mathbb Z^2$ killed on
leaving $\mathbb H$; equivalently, loops are nearest-neighbor loops contained
in $\mathbb H$, with Lupu's normalization giving a rooted loop of length $2n$
mass $(2n)^{-1}4^{-2n}$. This is the same graph-trace normalization as the one
obtained from the continuous-time walk after integrating out holding times.
Lupu \cite{Lupu2016a}, building on the critical non-percolation result from
\cite{Lupu2016}, proved that this half-plane loop soup has critical intensity
$1/2$; in particular, it percolates for every $\alpha>1/2$.

For the corresponding cable system over the killed half-plane, the threshold is
also $1/2$ in the same parametrization. Indeed, the critical
non-percolation at $\alpha=1/2$ follows from the metric graph Gaussian free
field result quoted above \cite{MR4421175}, via Lupu's coupling
\cite{Lupu2016}, while the supercritical side follows from the discrete
half-plane percolation result and the domination of the discrete loop soup by
the cable system loop soup. Thus, on $\mathbb H$, the discrete and cable system
thresholds agree:
\[
\widetilde \alpha_{\mathsf c}(\mathbb H)
= \alpha_{\mathsf c}(\mathbb H)
= \frac12.
\]
This half-plane example shows that the cable system enhancement need not lower
the critical point. The killing at the boundary is part of this example: it is
what makes the half-plane Green function finite and puts the model in the
metric graph/Gaussian-free-field framework. Without this killing one is back on
the full plane, where Lupu notes that the loop soup forms a single cluster for
every positive intensity, so there is no analogous non-trivial threshold to
compare. It would be interesting to develop a more general theory that separates
sensitive from non-sensitive loop soup or zoo-type models; we leave this for
future work.

\item
As a direct follow-up to the previous item, an interesting question is to understand whether $d\ge5$ is optimal or whether the strict inequality of \eqref{eq:strict} is valid for $G= \mathbb Z^3$ or $G= \mathbb Z^4$ already.
Indeed,  Chang and Sapozhnikov \cite{ChangSapo16} show that in  contrast to dimensions $d \ge 5$, in dimensions $3$ and $4$ the one-arm probability decays at a slower rate than it would if it was realized by a single large loop, which is an indication that 
connections are realized in a structurally different way in $G= \mathbb Z^3$ and $G= \mathbb Z^4$ when compared to $\mathbb Z^d$, $d \ge 5$ (cf.\ \cite{MR5067664}) From a different point of view the recent results of \cite{cai2026gapclusterdimensionsloop} also suggest that small (and mesoscopic) loops do contribute in the putative scaling limit of the random walk loop soup in $\mathbb Z^3$.


\item 

Analogous results had been
obtained for the Gaussian free field, where one knows that the critical
parameter for the percolation problem of the excursion sets of the metric graph
process satisfies $\widetilde{h}_*(G) = 0 $ \cite{Lupu2016,MR4421175} while the
critical parameter for the respective discrete problem satisfies $h_*(G) > 0$
\cite{DPR18,MR4914069} for $G = \mathbb Z^d$, and also more general graphs
$G$ with polynomial growth and regular Green function decay.
It also follows e.g.\ from \cite{DuGoRoSe-20} that the transition from finite expectation of the size of the cluster of the origin happens at the same threshold as the onset of percolation.

\end{enumerate}
\end{remark}

\begin{proof}[Proof of Theorem~\ref{thm: loop soup}]
\emph{Part a).} The proof relies on a natural comparison argument to a Bienaym\'e--Galton--Watson tree that has been used before, e.g.\ in
\cite{ChangSapo16}. The underlying idea is that if the one-animal generation has mean strictly smaller than
one, then the full cluster is dominated by a subcritical Bienaym\'e--Galton--Watson process.
More precisely, for \(k\ge0\) let \(C_\alpha^\mu(\0,k)\) be the vertex set of
the part of the cluster that can be reached from \(\0\) using a chain of at
most \(k\) different occupied animals, consecutive members of which have
intersecting vertex sets, with
\(C_\alpha^\mu(\0,0)=\{\0\}\). If
\[
M(\alpha):=
\mathbb E\big[|C_\alpha^\mu(\0,1)\setminus\{\0\}|\big]<1,
\]
then \(\chi_\mu(\alpha)<\infty\), and hence
\[
\alpha \leq \alpha_{\#}\le \alpha_{\mathsf c}.
\]

Indeed,
\[
V(C^\mu_\alpha(\0))=\bigcup_{k=0}^\infty C_\alpha^\mu(\0,k).
\]
Setting
\[
\Delta_k:=C_\alpha^\mu(\0,k)\setminus C_\alpha^\mu(\0,k-1),
\]
then we observe that discarding possible overlaps only makes the exploration smaller, and
the number \(|\Delta_k|\) is stochastically dominated by the sum of
\(|\Delta_{k-1}|\) independent copies of
\(|C_\alpha^\mu(\0,1)\setminus\{\0\}|\). Thus the generation sizes are dominated
by a Bienaym\'e--Galton--Watson process with offspring mean \(M(\alpha)<1\). The total
progeny has finite expectation, and therefore \(\chi_\mu(\alpha)<\infty\).

Since -- as mentioned above -- the cable system critical value is
\(\widetilde \alpha_c(\mathbb Z^d)=1/2\), \(d\ge3\), it remains to prove that
\begin{equation} \label{eq:subcritical condition}
M(\alpha):=\mathbb E\big[|C_\alpha^\mu(\0,1)\setminus\{\0\}|\big] < 1
\qquad
\text{ for some } \alpha > 1/2.
\end{equation}
 
The proof of Proposition 5.1 of \cite{ChangSapo16} provides the equality
\begin{equation} \label{eq:expClusterSize}
M(\alpha)
=
\sum_{v \in V \setminus \{\0\}}
1 - \big( 1 - q_G(v)^2 \big)^\alpha,
\end{equation}
where
\[
q_G(v) := \frac{\mathcal G_G(\0,v)}{\mathcal G_G(\0,\0)},
\]
and where $\mathcal G_G(\0,v)$ is the Green function
of the simple random walk on the graph $G$. The formula
is true once the random walk is transient on $G$; see also Proposition 18, Chapter 4 of \cite{LeJanStFlour}.

We will now show that 
\begin{equation} \label{eq:M12}
M(1/2)<1 \text{ for }\mathbb Z^d \text{ when }d\ge5.
\end{equation}
This will finish the proof since
$\alpha \mapsto M(\alpha)$
is continuous in a neighborhood of $1/2$: indeed,
$M(1/2) < 1$ implies $\sum_{v\neq o}q_G(v)^2<\infty$,
because
$1-\sqrt{1-t}\ge t/2$ for $t\in[0,1]$. For $\beta$ in a compact neighborhood
of $1/2$ there is a constant $C<\infty$ such that
\begin{equation} \label{eq:betaBd}
1-(1-t)^\beta\le Ct,
\qquad 0\le t\le1.
\end{equation}

Thus the terms defining $M(\beta)$ are dominated by $Cq_G(v)^2$, which as shown above is
summable under \eqref{eq:M12}. Dominated convergence then gives the claimed continuity.

We now prove \eqref{eq:M12}. The strategy is to first establish
\begin{equation} \label{eq:dimIneq}
    S_d\le S_5 \text{ for all }d\ge5,
\end{equation}
and afterwards explicitly compute that $M(1/2) < 1$ in $\mathbb Z^5$.

For $\mathbb Z^d$ with $d \geq 5$, let
\[
B_d:=
\sum_{x\neq0}\left(\frac{\mathcal G_{\mathbb Z^d}(0,x)}
{\mathcal G_{\mathbb Z^d}(0,0)}\right)^2.
\]
Since $1-\sqrt{1-t}\le t$ due to \eqref{eq:betaBd}, by \eqref{eq:expClusterSize} it is enough to prove $B_d<1$.
Writing
\[
S_d:=\sum_{x\in\mathbb Z^d}\mathcal G_{\mathbb Z^d}(0,x)^2
\]
we have
\[
B_d=\frac{S_d}{\mathcal G_{\mathbb Z^d}(0,0)^2}-1
\le S_d-1.
\]
In order to demonstrate \eqref{eq:dimIneq}, for $k \in \mathbb R^d$ we set
\[
\widehat p_d(k):=\frac1d\sum_{j=1}^d\cos k_j,
\]
so Parseval's identity gives
\[
S_d=\frac{1}{(2\pi)^d}\int_{[-\pi,\pi]^d}
\frac{\mathrm{d}k}{(1-\widehat p_d(k))^2}.
\]
Using
\[
\frac{1}{(1-r)^2}=\int_0^\infty t e^{-t(1-r)}\,\mathrm{d}t,
\qquad r<1,
\]
and Tonelli's theorem, this becomes
\begin{equation}\label{eq:Sd}
S_d
=
\int_0^\infty t e^{-t}
\frac{1}{(2\pi)^d}\int_{[-\pi,\pi]^d}
\exp\{t\widehat p_d(k)\}\,\mathrm{d}k\,\mathrm{d}t.
\end{equation}
Define
\[
I_0(s):=\frac1{2\pi}\int_{-\pi}^{\pi}e^{s\cos u}\,\mathrm{d}u
\]
for the modified Bessel function. Since the inner integral in \eqref{eq:Sd} factorizes over the
coordinates, we obtain
\[
S_d=\int_0^\infty t e^{-t} I_0(t/d)^d\,\mathrm{d}t.
\]
The function $\log I_0$ is convex and vanishes at the origin; hence
$s\mapsto \log I_0(s)/s$ is increasing on $(0,\infty)$. It follows that
$d\mapsto I_0(t/d)^d$ is decreasing for each fixed $t>0$. Thus
\eqref{eq:dimIneq}
follows.

It remains to upper bound $S_5$ by a finite exact calculation followed by an
explicit tail estimate. Letting $p_n^{(5)}(0)$ be the $n$-step return probability of simple random
walk on $\mathbb Z^5$, we can then rewrite
\[
S_5=\sum_{n,m\ge0}p_{n+m}^{(5)}(0)
=1+\sum_{m\ge1}(2m+1)p_{2m}^{(5)}(0).
\]
Exact evaluation
(see Appendix \ref{sec:exComp} also) gives
\begin{equation}\label{eq:finite-sum-bound}
1+\sum_{m=1}^{307}(2m+1)p_{2m}^{(5)}(0)
<1.88939.
\end{equation}
For the remaining tail we use the explicit return-probability bound of
\cite[Proposition~1 and the paragraph following Remark~1]{BallSterbenz05}.
That bound applies when
$m\ge \frac74d^2(d+2)+1$; in dimension $d=5$, and for integer $m$, it gives
\[
p_{2m}^{(5)}(0)
\le
2\left(\frac{5}{4\pi m}\right)^{5/2},
\qquad m\ge308.
\]
Since $(2t+1)t^{-5/2}$ is decreasing, the remaining tail is bounded by
\[
2\left(\frac{5}{4\pi}\right)^{5/2}
\,\int_{307}^\infty (2t+1)t^{-5/2}\,\mathrm{d}t
<0.0457.
\]
As a consequence, and in combination with the previously established
\eqref{eq:dimIneq}, we get $S_d\le S_5<1.936$ for all $d\ge5$, and hence
$B_d\le S_d-1<0.936<1$.

\emph{Part b).} We apply Theorem~\ref{thm:zenhan}, so we only need to verify its assumptions for the specific case of loop soups. Let $\nu$ be the non-trivial Markov loop measure from \eqref{eq:discreteLoopMeasure}, for $G=\mathbb T_d$, and let $\mu=\operatorname{tr}_{\#}\nu$ be its graph-animal push-forward from \eqref{eq: loop soup intensity}. Then
$\mu$ is invariant and non-degenerate, since every edge belongs to the trace
of a two-step backtracking loop with positive $\nu$-measure. We claim that the finiteness condition
\eqref{eq: finite C_ast} holds for this $\mu$. To prove this, let \(x\in V\), denote by \(\P_x\) the law of \(X\) started at \(x\), and set
\[
\tau_v:=\inf\{t\geq0:X_t=v\}.
\]
For \(u\neq v\), define
\[
F(u,v):=\P_u(\tau_v<\infty).
\] The standard two-point
loop-measure identity (cf.\ Lemma 2.5 in \cite{ChangSapo16}) gives
\[
W_{\mu}(u,v)
=
\sum_{A\in\mathfrak A:\,u,v\in V(A)}\mu(A)
=
\nu\bigl(\{\gamma\in\mathcal L^{\mathrm{nt}}:\,
u,v\in V(\operatorname{tr}(\gamma))\}\bigr)
=
-\log\bigl(1-F(u,v)F(v,u)\bigr).
\]
On $\mathbb T_d$, if $n=\operatorname{dist}(u,v)$, then
$F(u,v)=F(v,u)=(d-1)^{-n}$. Since $(d-1)^{-2n}\le1/4$ for $n\ge1$,
\[
W_{\mu}(u,v)\le 2(d-1)^{-2n}.
\]
The diagonal term $W_{\mu}(u,u)$ is finite by transience; here it is essential
that $\nu$ excludes the zero-jump loops. Summing
over the sphere of radius $n$, whose size is $d(d-1)^{n-1}$ for $n\ge1$, gives
\[
\sum_{v\in V(\mathbb T_d)}W_{\mu}(u,v)<\infty.
\]
By transitivity this is uniform in $u$, and hence
\eqref{eq: finite C_ast} holds for $\mu$. Therefore Theorem~\ref{thm:zenhan} implies the statement of
Theorem~\ref{thm: loop soup} \ref{item:treeStrict}.

For completeness, the dependence of $q_{\mathrm{cab}}$ on the intensity causes no
problem in the asserted threshold comparison. Set
$a:=\alpha_{\mathsf c}(\mathbb T_d)/2>0$ and
$\delta:=q_{\mathrm{cab}}(a)$. Infinitesimal sensitivity gives some
$b<\alpha_{\mathsf c}(\mathbb T_d)$ for which the level-$b$ discrete soup,
enhanced by iid Bernoulli edges of parameter $\delta$, percolates. At the cable
level $c:=\max\{a,b\}<\alpha_{\mathsf c}(\mathbb T_d)$, the monotone Poisson
coupling contains both the level-$b$ fundamental soup and, independently, the
field extracted above at level $a$. Consequently
\[
\widetilde\alpha_{\mathsf c}(\mathbb T_d)
\le c<\alpha_{\mathsf c}(\mathbb T_d)=\alpha_{\#}(\mathbb T_d),
\]
where the last equality is Theorem~\ref{thm:zsharp}.

\end{proof}

\section{Analysis of Poisson zoos on trees}\label{sec:trees}

The goal of this section is to prove Theorems~\ref{thm:zsharp} and
\ref{thm:zenhan}. The finite-volume part of Subsection \ref{sec:di} is stated for general transitive
graphs; the subsequent exploration (see Subsection \ref{sec:2124}) argument is specialized to
$G=\mathbb T_d$, $d\ge3$.

\subsection{Finite-volume differential inequality and openness} \label{sec:di}

The finite-volume estimate below is the analytic input. It controls how fast
the finite-volume susceptibility can grow as the intensity parameter increases.
The overlap constant $C_*(\mu)$ naturally appears as the quantity which bounds the total possible
influence of one cluster on another. The results in this particular subsection are
not specific to trees and hold on any transitive graph $G$.

\begin{lemma}[Monotone coupling]\label{lem:zmoncp}
There is a coupling of the Poisson zoos
\((\mathcal O^\mu_\alpha)_{\alpha\ge0}\) such that
\[
\mathcal O^\mu_{\alpha_1}\subseteq \mathcal O^\mu_{\alpha_2}
\qquad\mathbb{P}\text{-a.s. whenever }\alpha_1\le \alpha_2.
\]
Moreover, if \(\alpha \in (0,\infty)\) and \(\alpha_n\uparrow \alpha\), then
\begin{equation} \label{eq:union}
C^\mu_\alpha(\0)
=
\bigcup_{n\ge1} C^\mu_{\alpha_n}(\0)
\qquad\mathbb{P}\text{-a.s.}
\end{equation}
Consequently,
\[
\chi_\mu(\alpha)=\lim_{n\to\infty}\chi_\mu(\alpha_n).
\]
\end{lemma}

\begin{proof}
Let \(\Pi\) be a Poisson point process on
\(\mathfrak A\times\mathbb R_+\) with intensity
\(\mu(A)\otimes dt\). We declare an animal \(A\) to be present at level
\(\alpha\) if \((A,t)\in\Pi\) for some \(t\le\alpha\). This gives the
monotone coupling.

Since
\[
\Pi(\mathfrak A\times\{\alpha\})=0
\]
almost surely for deterministic \(\alpha\), the configuration at level
\(\alpha\) is the increasing union of the configurations at levels
\(\alpha_n\uparrow\alpha\). The identity \eqref{eq:union} follows, and the statement for
\(\chi_\mu\) follows by monotone convergence.
\end{proof}

\begin{lemma}[Aizenman--Newman inequality for Poisson zoos]
\label{lem:zan}
Assume $C_*(\mu)<\infty$. Let $(V_n)_{n\ge1}$ be a
increasing finite connected
exhaustion of $V(G)$ with $\0\in V_1$, and let
\[
\widehat\chi_n^\mu(\alpha)
:=
\max_{x\in V_n}\sum_{y\in V_n}\P^{(n)}(x\leftrightarrow y \text{ at level }\alpha),
\]
where $\P^{(n)}$ is the probability measure underlying the finite-volume zoo
using only animals $A\in\mathfrak A$ with $V(A)\subseteq V_n$, and
$\{x\leftrightarrow y \text{ at level }\alpha\}$ denotes the event of $x$ and $y$ being contained in the same connected component of the graph induced by the
corresponding occupied edge set at level $\alpha$. Then for every $n$ and Lebesgue-a.e.\
$\alpha\ge0$,
\begin{equation} \label{eq:diffIneq}
\frac{\mathrm d}{\mathrm d\alpha}\widehat\chi_n^\mu(\alpha)
\le
C_*(\mu)\,\widehat\chi_n^\mu(\alpha)^2.
\end{equation}
\end{lemma}

\begin{proof}
For $x,y\in V_n$, write
$p_n^\alpha(x,y):=\P^{(n)}(x\leftrightarrow y \text{ at level }\alpha)$, and write
$C_\alpha^{(n)}(x)$ for the finite-volume cluster of $x$. In finite volume the
animal counts are independent Poisson variables. The connection events depend
only on the associated occupation indicators, and the BK inequality
\cite{vdBK1985} applies to these independent indicator coordinates for
increasing events with disjoint animal-coordinate witnesses; see also
\cite{vdB1996PoissonBK} for the corresponding Poisson setting. The
differentiation in $\alpha$ is given by the Poisson Russo--Mecke formula, i.e.\
the Mecke formula applied to the intensity parameter; see
\cite{LastPenrose2018}.

Thus, for $x\ne y$ and a.e.\ $\alpha$,
\begin{align*}
\frac{\mathrm d}{\mathrm d\alpha}p_n^\alpha(x,y)
&=
\sum_{\substack{A\in\mathfrak A:\,V(A)\subseteq V_n}}\mu(A)\,
\P^{(n)}\left(
\begin{gathered}
C_\alpha^{(n)}(x)\cap C_\alpha^{(n)}(y)=\varnothing,\\
V(A)\cap C_\alpha^{(n)}(x)\ne\varnothing,\\
V(A)\cap C_\alpha^{(n)}(y)\ne\varnothing
\end{gathered}
\right).
\end{align*}
On the event in the probability, choose $u,v\in V(A)$ with
$x\leftrightarrow u$ and $y\leftrightarrow v$. Since the two clusters are
disjoint, these two connections occur disjointly, and BK gives
\[
\frac{\mathrm d}{\mathrm d\alpha}p_n^\alpha(x,y)
\le
\sum_{u,v\in V_n}
p_n^\alpha(x,u)\,p_n^\alpha(y,v)\,W_\mu(u,v).
\]
For $x=y$ the same inequality is immediate. Set
\[
\chi_{n,x}^\mu(\alpha):=\sum_{y\in V_n}p_n^\alpha(x,y).
\]
Summing over $y$ and using symmetry of connection probabilities gives, for
each $x\in V_n$,
\begin{align*}
\frac{\mathrm d}{\mathrm d\alpha}
\chi_{n,x}^\mu(\alpha)
&\le
\sum_{u\in V_n}p_n^\alpha(x,u)
\sum_{v\in V_n}W_\mu(u,v)\sum_{y\in V_n}p_n^\alpha(v,y)\\
&\le
C_*(\mu)\,\widehat\chi_n^\mu(\alpha)^2.
\end{align*}
Since $\widehat\chi_n^\mu$ is the maximum of finitely many absolutely continuous
functions, it is absolutely continuous. At a.e.\ $\alpha$ where all
$\chi_{n,x}^\mu$ and $\widehat\chi_n^\mu$ are differentiable, the elementary
derivative rule for finite maxima gives
\[
\big(\widehat\chi_n^\mu\big)'(\alpha)
\le
\max_{x:\,\chi_{n,x}^\mu(\alpha)=\widehat\chi_n^\mu(\alpha)}
\big(\chi_{n,x}^\mu\big)'(\alpha).
\]
Here the maximum is taken over the vertices \(x\) at which
\(\chi_{n,x}^\mu(\alpha)\) attains the value
\(\widehat\chi_n^\mu(\alpha)\). The preceding bound holds for each such
vertex, and therefore yields the stated inequality.
\end{proof}

The differential inequality implies that finite susceptibility cannot break
down instantaneously. Starting from a parameter $\alpha_0$ with finite
susceptibility, integrating the inequality for the inverse susceptibility gives
a uniform bound slightly to the right of $\alpha_0$. This is the usual
\emph{openness of the subcritical phase} mechanism (when defined in terms of finite susceptibility, opposed to the percolation probability).

\begin{lemma}[Openness of the subcritical phase]\label{lem:zopdiv}
Assume $\mu$ is such that $C_*(\mu)<\infty$. If $\chi_\mu(\alpha_0)<\infty$ for some $\alpha_0 \in (0,\infty)$, then there exists
$\delta>0$ such that
\begin{equation} \label{eq:finiteness}
\chi_\mu(\alpha)<\infty
\qquad\text{for all }\alpha\in[\alpha_0,\alpha_0+\delta).
\end{equation}
In particular, if $\alpha_{\#}^\mu<\infty$, then
\[
\lim_{\alpha\uparrow\alpha_{\#}^\mu}\chi_\mu(\alpha)=\infty.
\]
\end{lemma}

\begin{proof}
Let $\widehat\chi_n^\mu$ be as in Lemma~\ref{lem:zan}. Finite-volume
monotonicity and transitivity give
\[
\widehat\chi_n^\mu(\alpha_0)\le\chi_\mu(\alpha_0)
\qquad\text{for all }n.
\]
For fixed $n$, the quantity $\widehat\chi_n^\mu(\alpha)$ is finite and positive
for all $\alpha$, and it is absolutely continuous on compact intervals. Hence
$(\widehat\chi_n^\mu)^{-1}$ is absolutely continuous, and at a.e.\ point of
differentiability of $\widehat\chi_n^\mu$ the differential inequality
\eqref{eq:diffIneq} gives
\[
\frac{\mathrm d}{\mathrm d\alpha}(\widehat\chi_n^\mu(\alpha))^{-1}
\ge -C_*(\mu).
\]
No differentiability at the endpoint $\alpha_0$ is needed. Integrating this
a.e.\ inequality on $[\alpha_0,\alpha]$ and using the preceding initial bound
yields
\[
(\widehat\chi_n^\mu(\alpha))^{-1}
\ge
\chi_\mu(\alpha_0)^{-1}-C_*(\mu)(\alpha-\alpha_0).
\]
Choosing
\[
\delta:=\frac{1}{2C_*(\mu)\chi_\mu(\alpha_0)}
\]
(with the evident interpretation if $C_*(\mu)=0$) gives
$\widehat\chi_n^\mu(\alpha)\le2\chi_\mu(\alpha_0)$ for all
$\alpha\in[\alpha_0,\alpha_0+\delta)$, uniformly in $n$. Letting
$n\to\infty$ along the exhaustion provides us with \eqref{eq:finiteness}.

If the left limit of $\chi_\mu$ at a finite $\alpha_{\#}^\mu$ were
finite, Lemma~\ref{lem:zmoncp} would give
$\chi_\mu(\alpha_{\#}^\mu)<\infty$, and openness at
$\alpha_{\#}^\mu$ would contradict the definition of
$\alpha_{\#}^\mu$.

The same integrated estimate also gives the usual mean-field lower bound on
the divergence of the susceptibility. Indeed, for
$\alpha<\beta<\alpha_{\#}^\mu$, applying the preceding argument on
$[\alpha,\beta]$ and then letting $n\to\infty$ yields
\[
(\chi_\mu(\beta))^{-1}
\ge
(\chi_\mu(\alpha))^{-1}-C_*(\mu)(\beta-\alpha).
\]
Letting $\beta\uparrow\alpha_{\#}^\mu$ and using the divergence just proved
gives
\[
\chi_\mu(\alpha)\ge
\frac{1}{C_*(\mu)(\alpha_{\#}^\mu-\alpha)}
\qquad(\alpha<\alpha_{\#}^\mu),
\]
when $0<C_*(\mu)<\infty$ and $\alpha_{\#}^\mu<\infty$. Thus, in cases where
$\alpha_{\#}^\mu=\alpha_{\mathsf c}^\mu$, the susceptibility exponent, if it
exists, is at least one.
\end{proof}

\subsection{Exploration scheme and sensitivity for the $d$-regular tree} \label{sec:2124}

From now on, fix $d\ge3$ and set $G=\mathbb T_d$. We now prove the two tree
statements, Theorems~\ref{thm:zsharp} and
\ref{thm:zenhan}. The proof of the coincidence of critical values is the
direct boundary-sprinkling argument; the enhancement theorem will then use this
coincidence result.
The basic idea, following the exploration scheme of Chang and Sapozhnikov
\cite{ChangSapo16}, can in principle be applied to all graphs. Let $\alpha <
\alpha_{\mathsf c}^\mu(G)$,
so that $C_\alpha^\mu(\0)$ is finite almost surely.
We can build this cluster by the following exploration scheme:
we first sample the Poisson counts $N_A$ of animals for all
$A \in \mathfrak A$ that contain $\0$. Then we sample $N_A$ for all
$A \in \mathfrak A$ that contain a vertex covered by this first group of
animals, but do not contain $\0$. We continue recursively by sampling
$N_A$ for those $A$ that contain a vertex that has been uncovered
in the previous generation, but no vertex from older generations. The algorithm
terminates once there is an iteration where no new vertex is discovered. Since $\alpha < \alpha_{\mathsf c}^\mu(G)$,
it does so almost surely. The cluster of discovered vertices has the same distribution
as $C_\alpha^\mu(\0)$.

Now consider the edge boundary $\partial_E C_\alpha^\mu(\0)$  of this cluster,
i.e.\ the edges of $G$ that have one vertex in $C_\alpha^\mu(\0)$ and one
outside. Since $C_\alpha^\mu(\0)$ is a finite connected vertex set in
$\mathbb T_d$, the tree identity gives
\[
|\partial_E C_\alpha^\mu(\0)|
=(d-2)|C_\alpha^\mu(\0)|+2.
\]
Thus Lemma~\ref{lem:zopdiv} implies that, when
$\alpha_{\#}^\mu<\infty$, the expected boundary size diverges as
$\alpha\uparrow\alpha_{\#}^\mu$. In particular, for each $\delta\in(0,1]$ and
$M<\infty$, we may choose
$\alpha<\alpha_{\#}^\mu\leq\alpha_{\mathsf c}^\mu$ so that
$\mathbb E[|\partial_E C_\alpha^\mu(\0)|]>M/\delta$.
Now a Bernoulli enhancement of strength $\delta$ is applied to the edges of
$\partial_E C_\alpha^\mu(\0)$, creating on average at least $M$
open edges that lead to vertices that are not in $C_\alpha^\mu(\0)$.
Let us call such vertices {\em seed vertices}.

At this point, we would like to repeat the procedure recursively: let
$v \in G \setminus C_\alpha^\mu(\0)$ be a seed vertex. We can then
grow the connected component corresponding to $v$ just as before: let
$\mathfrak A_0$ be the set of all graph animals that contain no vertex in
$C_\alpha^\mu(\0)$. These are precisely those animals where $N_A$ has not
yet been sampled. We then start
by sampling $N_A$ for all $A \in \mathfrak A_0$ that contain $v$, then for those
that contain a vertex just discovered but not $v$, and so on. We end up with
a cluster $\widetilde C_{\alpha}^{\mu}(v)$. When we do this for every seed vertex,
we have thus explored a larger part of $G$, still finite. This new part
again has an edge boundary, on which we apply the Bernoulli sprinkling,
creating next generation seeds, and so on. This algorithm stops when either
no seeds are produced at a certain point, or when the exploration in a certain
generation yields $N_A = 0$ for all graph animals containing any of the seeds.
However, if each seed $v$ produces on average at least a number $s(v)>1$ of new
seeds, then we can couple the number of total seeds in generation $n$ to the
total offspring of a supercritical Bienaym\'e--Galton--Watson tree, which shows
that the Poisson zoo is $\delta$-sensitive to Bernoulli enhancement for this
particular $\delta$.

There are two major problems with this approach. The first is that
even when $v$ is mapped to $\0$ by
a graph isomorphism, the distribution of $\widetilde C_{\alpha}^{\mu}(\0)$
is different from the distribution of $C_\alpha^\mu(\0)$: the
algorithm only has access to those graph animals that do not intersect
the previously explored area, so $\widetilde C_{\alpha}^{\mu}(\0)$ will be
(possibly much) smaller than $C_{\alpha}^{\mu}(\0)$.
The second problem is that
during the exploration of $\widetilde C_{\alpha}^{\mu}(v)$ for a given seed $v$,
a possibly large number of other seed vertices is already discovered. So, even
finding
a large number of seed vertices does not guarantee that we will have a large
number of attempts at building the next generation cluster.

Both problems seem to be rather hard on $\mathbb Z^d$,
but are solvable on the tree.
Indeed, the second problem is completely absent there, because at each step,
the explored area is a connected subgraph of $\mathbb T_d$, and so two
different seeds are necessarily in different connected components of the
unexplored area.

For the first problem, we introduce the following notions: for $x,y \in \mathbb T_d$ write $d(x,y)$ for the
graph distance of $x$ and $y$. Let $e = (x,y)$ be an oriented edge of
$\mathbb T_d$, so that $d(\0,x) < d(\0,y)$.
$T_e^+$ denotes the connected component of
$\mathbb T_d\setminus\{e\}$ containing $y$, and $T_e^-$  the connected component containing $x$. Let
\[
\mathfrak A_e^\pm := \{A \in \mathfrak A: V(A) \subset V(T_e^\pm) \}
\]
be the set of animals completely contained in $T_e^\pm$, and
\[
\mathcal O_{\alpha}^{\mu,\pm,e} := \bigcup_{A \in \mathfrak A_e^\pm : N_A^\alpha
\geq 1} E(A).
\]
Write $C_{\alpha}^{\mu,+,e}(y)$ for the cluster of
$\mathcal O_{\alpha}^{\mu,+,e}$ that contains $y$,
and $C_{\alpha}^{\mu,-,e}(x)$ for the cluster of
$\mathcal O_{\alpha}^{\mu,-,e}$ that contains $x$.

Finally, let $\mathfrak A_e := \mathfrak A \setminus (\mathfrak A_e^+ \cup \mathfrak A_e^-)$ be the set of animals that contain the edge $e$, and call
\[
\lambda_\mu:=\sum_{A\in\mathfrak A_e}\mu(A)
\]
the edge intensity of the Poisson zoo. By the invariance of $\mu$, it
does not depend on
the choice of $e$, and unless the zoo is empty with probability one (a case that
we tacitly exclude), we have $\lambda_\mu > 0$ if the zoo is non-degenerate.

The factor $e^{-\alpha\lambda_\mu}$ below has a simple meaning:
for an edge to be a boundary edge of the root cluster, the root must be
connected to the near endpoint, but no occupied animal may use the boundary edge
itself. Since animals are connected, any animal crossing from one side of the
edge to the other must contain that edge. This makes the no-crossing event
independent of the inward connection event.

\begin{lemma}\label{lem:zbfact}
In the setting $G=\mathbb T_d$, $d\ge3$, for every
$\alpha < \alpha_{\mathsf c}^\mu(G)$ and every oriented edge $e=(x_e,y_e)$
directed away from $\0$, so that $y_e$ is the endpoint farther from $\0$, we have
\[
\E[|\partial_E C_\alpha^\mu(\0)|]
=
d\,e^{-\alpha\lambda_\mu}\,
\E[|C_{\alpha}^{\mu,+,e}(y_e)|].
\]
The equality is understood in $[0,\infty]$.
\end{lemma}

\begin{proof}
For an oriented edge $f=(u,v)$,
the event $\{f\in\partial_E C_\alpha^\mu(\0)\}$ is the
intersection of
\[
\{\0 \stackrel{\mathfrak A_f^-}{\longleftrightarrow} u \text{ at level }\alpha\} :=
\{\0\leftrightarrow u\text{ using only animals } A \in \mathfrak A_f^- \text{ with } N_A \ge 1\},
\]
with the event that $N_A = 0$ for all $A \in \mathfrak A_f$. These two
events depend on disjoint animal classes and are therefore independent. The
second has probability $\exp(-\alpha\lambda_\mu)$. Hence
\[
\P(f\in\partial_E C_\alpha^\mu(\0))
=
e^{-\alpha\lambda_\mu}
\P(\0 \stackrel{\mathfrak A_f^-}{\longleftrightarrow} u \text{ at level } \alpha).
\]
Since $\alpha<\alpha_{\mathsf c}^\mu(G)$, the cluster
$C_\alpha^\mu(\0)$ is finite almost surely. Thus every boundary edge of
$C_\alpha^\mu(\0)$ is directed away from $\0$. By Tonelli's theorem,
\[
\mathbb E[|\partial_E C_\alpha^\mu(\0)|]
=
\sum_{f=(x,y)\text{ directed away from }\0}
\P(f\in\partial_E C_\alpha^\mu(\0)).
\]
Combining this with the previous equality gives
\[
\mathbb E[|\partial_E C_\alpha^\mu(\0)|]
=
e^{-\alpha\lambda_\mu}
\sum_{f=(u,v)\text{ directed away from }\0}
\P(\0 \stackrel{\mathfrak A_f^-}{\longleftrightarrow} u \text{ at level }\alpha).
\]
It remains to identify the last sum. Fix the oriented edge
$e=(x_e,y_e)$ from the statement. If $f=(u,v)$ is directed away from $\0$ and
$d(\0,u)=k$, then there is a tree automorphism mapping the component
$T_f^-$, viewed from the boundary endpoint $u$, to the forward component
$T_e^+$, viewed from $y_e$, and sending the distinguished vertex $\0$ to some
vertex $w\in T_e^+$ with $d(y_e,w)=k$. By invariance, the corresponding
connection probability is the probability that $w$ belongs to
$C_{\alpha}^{\mu,+,e}(y_e)$. There are $d(d-1)^k$ oriented edges
$f=(u,v)$ directed away from $\0$ whose endpoint $u$ closer to $\0$ satisfies
$d(\0,u)=k$, whereas at distance $k$ from $y_e$ inside $T_e^+$ there are
$(d-1)^k$ vertices. Thus, level by level, the boundary-edge sum has a factor
$d$ relative to the forward-cluster expectation, and therefore
\[
\sum_{f=(u,v)\text{ directed away from }\0}
\P(\0 \stackrel{\mathfrak A_f^-}{\longleftrightarrow} u \text{ at level }\alpha)
=
d\mathbb E[|C_{\alpha}^{\mu,+,e}(y_e)|].
\]
This proves the claim.
\end{proof}

For an oriented edge $e=(x,y)$ directed away from $\0$, define the forward
boundary
\[
\partial_E^+ C_{\alpha}^{\mu,+,e}(y)
:=
\{\{u,v\}:u\in C_{\alpha}^{\mu,+,e}(y),\
v\in T_e^+\setminus C_{\alpha}^{\mu,+,e}(y),\
v\text{ is farther from }x\text{ than }u\}
\]
and set
\[
b_\mu^+(\alpha)
:=
\E\bigl[|\partial_E^+ C_{\alpha}^{\mu,+,e}(y)|\bigr].
\]
The value of $b_\mu^+(\alpha)$ does not depend on the chosen oriented edge.
For $\alpha<\alpha_{\#}^\mu$, the forward cluster has finite expectation by
Lemma~\ref{lem:zbfact} and the tree identity
$|\partial_E C_\alpha^\mu(\0)|=(d-2)|C_\alpha^\mu(\0)|+2$. In particular it is
finite a.s., and hence
\[
|\partial_E^+ C_{\alpha}^{\mu,+,e}(y)|
=
(d-2)|C_{\alpha}^{\mu,+,e}(y)|+1.
\]
Combining this with Lemma~\ref{lem:zbfact} and the tree identity
$|\partial_E C_\alpha^\mu(\0)|=(d-2)|C_\alpha^\mu(\0)|+2$ gives
\[
b_\mu^+(\alpha)
=
\frac{d-2}{d}e^{\alpha\lambda_\mu}
\bigl((d-2)\chi_\mu(\alpha)+2\bigr)+1.
\]
Therefore Lemma~\ref{lem:zopdiv} implies
\[
b_\mu^+(\alpha)\xrightarrow[\alpha\uparrow\alpha_{\#}^\mu]{}\infty
\qquad
\text{whenever }\alpha_{\#}^\mu<\infty.
\]

\begin{lemma}\label{lem:zbridgefield}
Assume that \eqref{eq: finite C_ast} holds and that $\mu$ is non-degenerate.
For every oriented edge $e=(x,y)$ directed away from $\0$, let
\[
\mathfrak B_e
:=
\{A\in\mathfrak A:\ e\in E(A),\ V(A)\subseteq T_e^+\cup\{x\}\}.
\]
Then
\[
\lambda_\mu^+:=\mu(\mathfrak B_e)\in(0,\infty)
\]
does not depend on $e$, and the classes $\mathfrak B_e$, indexed by oriented
edges away from $\0$, are pairwise disjoint.
Consequently, for any $\varepsilon>0$, an independent zoo of intensity
$\varepsilon\mu$ defines, for each oriented edge $e$ away from $\0$, the
indicator
\[
\eta_e^\varepsilon
:=
\mathbf 1\{\text{at least one animal from }\mathfrak B_e
\text{ is present in the increment zoo}\}.
\]
The variables $(\eta_e^\varepsilon)$ are i.i.d.\ Bernoulli variables with parameter
\[
q_\varepsilon:=1-e^{-\varepsilon\lambda_\mu^+}>0.
\]
We call this i.i.d.\ family a {\em Bernoulli bridge field}. On the event
\(\eta_e^\varepsilon=1\), the increment zoo contains an occupied animal that
crosses the corresponding edge \(e\).
\end{lemma}

\begin{proof}
Finiteness follows from \eqref{eq: finite C_ast}, since every animal in
$\mathfrak B_e$ contains the endpoint $x$ of $e=(x,y)$. The automorphism
invariance of $\mu$ gives independence of the value of $\lambda_\mu^+$ from the
choice of $e$.
For positivity, choose an animal $A$ with $\mu(A)>0$ and at least one edge.
Since $A$ is a finite tree, it has a leaf edge. By an automorphism of
$\mathbb T_d$, map this leaf edge to $e=(x,y)$ with the leaf endpoint sent to
$x$ and the rest of $A$ sent into $T_e^+$. This image has positive $\mu$-mass
and belongs to $\mathfrak B_e$.
Finally, an animal in $\mathfrak B_e$ has $e$ as its unique edge closest to
$\0$, so it cannot belong to $\mathfrak B_f$ for another oriented edge $f$.
Poisson restriction to the disjoint classes $\mathfrak B_e$ gives independence
of the indicators \(\eta_e^\varepsilon\), and the Poisson count in each class
has mean \(\varepsilon\lambda_\mu^+\). Thus each indicator has success
probability \(q_\varepsilon\), and every successful bridge contains the
corresponding edge.
\end{proof}

The same boundary estimate also applies in the hypothetical interval between
$\alpha_{\#}^\mu$ and $\alpha_{\mathsf c}^\mu$. Indeed, assume that $\mu$ is
non-degenerate and fix $\alpha<\alpha_{\mathsf c}^\mu$. Then the forward
cluster $C_{\alpha}^{\mu,+,e}(y)$ is finite a.s. For $\alpha=0$ this is
trivial. For $\alpha>0$, if it were infinite with positive probability for a
child edge $e=(\0,y)$, then the independent event that the level-$\alpha$ zoo
contains an animal from $\mathfrak B_e$ would connect $\0$ to this infinite
forward cluster with positive probability, contradicting
$\alpha<\alpha_{\mathsf c}^\mu$. Consequently, if
$\alpha_{\#}^\mu<\alpha<\alpha_{\mathsf c}^\mu$, then
\[
\E[|C_{\alpha}^{\mu,+,e}(y)|]=\infty,
\]
by Lemma~\ref{lem:zbfact}, because the full root cluster is finite a.s.\ but
has infinite expectation. Since the forward cluster is finite a.s., the forward
tree boundary identity applies and gives
\[
b_\mu^+(\alpha)=\infty.
\]

\begin{proposition}\label{prop:zforward-sprinkling}
Fix $\alpha<\alpha_{\mathsf c}^\mu$ and $q\in(0,1]$. Independently of the
level-$\alpha$ zoo, attach to each oriented edge away from $\0$ an independent
Bernoulli bridge variable with parameter $q$, and assume that a successful
bridge connects the two endpoints of the corresponding edge. If
\[
q\,b_\mu^+(\alpha)>1,
\]
then the union of the level-$\alpha$ zoo and the bridge field percolates with
positive probability.
\end{proposition}

\begin{proof}
If the forward cluster has positive probability to be infinite, then the
conclusion follows immediately after the first successful bridge. We may
therefore assume that all forward clusters below are finite a.s.

Start with a fixed child edge $e_0=(\0,y_0)$. With probability $q>0$, the
bridge across $e_0$ is present. Conditional on this event, reveal the zoo in
$T_{e_0}^+$ and take the cluster
$C_{\alpha}^{\mu,+,e_0}(y_0)$. Then reveal the bridge variables on its forward
boundary and continue recursively through every successful boundary edge.

The forward subtrees encountered by this exploration are pairwise disjoint.
Poisson restriction therefore gives independent copies of the same forward
cluster law, and the bridge variables on their forward boundaries are
independent Bernoulli variables with parameter $q$. The active boundary edges
are thus a Bienaym\'e--Galton--Watson process whose offspring variable is,
conditionally on the forward cluster,
\[
\operatorname{Binomial}\bigl(
|\partial_E^+C_{\alpha}^{\mu,+,e}(y)|,q
\bigr),
\]
and whose mean is $q\,b_\mu^+(\alpha)>1$. Hence this process survives with
positive probability. On survival, concatenating the revealed zoo clusters and
the successful bridge edges gives an infinite occupied path.
\end{proof}

\begin{proof}[Proof of Theorem~\ref{thm:zsharp}]

We now give the direct boundary-sprinkling argument. If the zoo has no positive
mass on any animal with an edge, then no edge is ever occupied and
$\alpha_{\mathsf c}^\mu=\alpha_{\#}^\mu=\infty$. We may therefore assume that
$\mu$ is non-degenerate. If $\alpha_{\#}^\mu=\infty$, then the elementary
inequality $\alpha_{\mathsf c}^\mu\ge\alpha_{\#}^\mu$ already gives equality.

Assume for contradiction that
$\alpha_{\#}^\mu<\alpha_{\mathsf c}^\mu$. Choose
$\bar\alpha\in(\alpha_{\#}^\mu,\alpha_{\mathsf c}^\mu)$ and then
$\delta>0$ such that
\[
\bar\alpha+\delta<\alpha_{\mathsf c}^\mu.
\]
Couple the zoo at level $\bar\alpha+\delta$ as the union of an independent
level-$\bar\alpha$ zoo and an independent increment zoo of intensity
$\delta\mu$. By Lemma~\ref{lem:zbridgefield}, this increment contains an iid
Bernoulli bridge field on the oriented edges away from $\0$, with parameter
\[
q_\delta:=1-e^{-\delta\lambda_\mu^+}>0.
\]
Since $\bar\alpha>\alpha_{\#}^\mu$, we have
$\chi_\mu(\bar\alpha)=\infty$. Since
$\bar\alpha<\alpha_{\mathsf c}^\mu$, the preceding boundary discussion gives
$b_\mu^+(\bar\alpha)=\infty$. Hence
\[
q_\delta b_\mu^+(\bar\alpha)>1.
\]
Proposition~\ref{prop:zforward-sprinkling} implies that the union of the
level-$\bar\alpha$ zoo and the increment bridge field percolates with positive
probability. This union is contained in the zoo at level
$\bar\alpha+\delta$, contradicting
$\bar\alpha+\delta<\alpha_{\mathsf c}^\mu$. Therefore
$\alpha_{\mathsf c}^\mu\le\alpha_{\#}^\mu$, and the reverse inequality is
immediate from the definitions.
\end{proof}

\begin{proof}[Proof of Theorem~\ref{thm:zenhan}]

We first verify that the critical point is positive. Let
\[
D_{\alpha}^\mu(\0)
:=
\bigcup_{\substack{A\in\mathfrak A:\,\0\in V(A)\\N_A^\alpha\ge1}}
\bigl(V(A)\setminus\{\0\}\bigr)
\]
be the set reachable from $\0$ using one occupied animal. By the union bound
and $1-e^{-t}\le t$,
\begin{align*}
\E\bigl[|D_{\alpha}^\mu(\0)|\bigr]
&\le
\sum_{v\ne\0}\sum_{\substack{A\in\mathfrak A:\,\0,v\in V(A)}}
\P(N_A^\alpha\ge1)\\
&\le
\alpha\sum_{A\in\mathfrak A:\,\0\in V(A)}
\bigl(|V(A)|-1\bigr)\mu(A)
\le \alpha C_*(\mu).
\end{align*}
Since the zoo is non-degenerate, $0<C_*(\mu)<\infty$. Choose
$\alpha_0>0$ so that $\alpha_0C_*(\mu)<1$. In the usual one-animal
exploration, completing the as-yet unrevealed animal classes to independent
full copies shows that $|C_{\alpha_0}^\mu(\0)|$ is stochastically dominated by
the total progeny of a Bienaym\'e--Galton--Watson process with offspring law
$|D_{\alpha_0}^\mu(\0)|$. Its offspring mean is strictly smaller than one, and
hence $\chi_\mu(\alpha_0)<\infty$. Consequently,
\[
\alpha_{\mathsf c}^\mu\ge\alpha_{\#}^\mu\ge\alpha_0>0.
\]

First note that the critical point is finite. Indeed, since the zoo is
non-degenerate, Lemma~\ref{lem:zbridgefield} applied to the level-$\beta$ zoo
itself gives an iid Bernoulli bridge field with parameter
\(1-e^{-\beta\lambda_\mu^+}\). For some finite \(\beta\) this parameter is
larger than the Bernoulli bond percolation threshold of \(\mathbb T_d\), and
the zoo at level \(\beta\) percolates. Hence
\(\alpha_{\mathsf c}^\mu<\infty\), and therefore
\(\alpha_{\#}^\mu<\infty\) by Theorem~\ref{thm:zsharp}. This high-intensity
observation is only used to prove finiteness of the threshold; the base
intensity used below will be chosen subcritical.

Fix a Bernoulli enhancement parameter $\delta\in(0,1]$. By the divergence of
$b_\mu^+(\alpha)$, choose $\alpha<\alpha_{\#}^\mu$ so close to
$\alpha_{\#}^\mu$ that
\[
\delta\,b_\mu^+(\alpha)>1.
\]
The independent Bernoulli enhancement is precisely a bridge field with
parameter $\delta$, so Proposition~\ref{prop:zforward-sprinkling} gives
percolation for the enhanced model at the base level $\alpha$. Since
$\alpha<\alpha_{\#}^\mu=\alpha_{\mathsf c}^\mu$ by
Theorem~\ref{thm:zsharp}, this proves $\delta$-sensitivity. As
$\delta\in(0,1]$ was arbitrary, the zoo is infinitesimally sensitive.
\end{proof}

\subsection{Proof of subcritical exponential tails for zoos}
\label{sec:proof-zoo-exponential-tails}

We now prove the exponential tail statement. The idea is different from the
proof of the coincidence of critical values. Instead of using the tree
boundary, we explore the cluster in blocks with respect to the auxiliary zoo
graph, where two vertices are adjacent if they lie in a common animal occupied
at level $\alpha$.
Finite susceptibility allows us to choose a block radius $K$ so that the
expected number of vertices at zoo-distance $K$ is less than one. The
exponential animal-tail assumption then gives exponential moments for the block
size, and the whole exploration is dominated by a subcritical weighted BGW
process.

The proof turns finite susceptibility into a subcritical block exploration. The
shells $S_k(\alpha)$ are measured in the auxiliary zoo graph rather than in the
original tree distance. Since the expected total cluster size is finite, one can
find a shell whose expected size is less than one. This shell becomes the
offspring generation of the block process, while the vertices discovered before
that shell form the block weight.

\begin{lemma}\label{lem:wgwexp}
Let \((B,Y)\) be a pair of non-negative integer-valued random variables such
that
\[
\E[Y]<1
\qquad\text{and}\qquad
\E[e^{\lambda_0(B+Y)}]<\infty
\]
for some \(\lambda_0>0\). Consider a Bienaym\'e--Galton--Watson process in
which each particle \(u\) carries an independent copy \((B_u,Y_u)\) of
\((B,Y)\), has \(Y_u\) children, and contributes weight \(B_u\). Then the total
weight
\[
W:=\sum_u B_u
\]
has an exponential tail. That is, there exist constants \(C,c>0\) such that
\[
\P(W\ge n)\le C e^{-cn}
\qquad\text{for all }n\ge1.
\]
\end{lemma}

\begin{proof}
Let \(f(s):=\E[s^Y]\). Since \(\E[Y]<1\) and \(Y\) has an exponential moment,
we can choose \(s>1\) close enough to \(1\) such that
\[
f(s)<s.
\]
By the exponential moment assumption and continuity at \(\theta=0\), after
decreasing \(\theta>0\) if necessary, we also have
\[
\E[e^{\theta B}s^Y]\le s.
\]
Let \(W_N\) be the total weight up to generation \(N\), and set
\[
\phi_N:=\E[e^{\theta W_N}].
\]
With the convention \(\phi_{-1}=1\), the branching property gives
\[
\phi_N=\E[e^{\theta B}\phi_{N-1}^Y].
\]
Since \(\phi_{-1}=1\le s\), induction yields \(\phi_N\le s\) for every \(N\).
By monotone convergence,
\[
\E[e^{\theta W}]\le s.
\]
Markov's inequality gives
\[
\P(W\ge n)\le s e^{-\theta n},
\]
which proves the claim.
\end{proof}

\begin{proof}[Proof of Theorem~\ref{thm:zsharp} \ref{item:exponentialCrit}]
For $\alpha=0$ the claim is immediate, so assume $\alpha>0$ and
$\chi_\mu(\alpha)<\infty$. Let $\mathcal Z^\mu_\alpha$ be the auxiliary graph on
$V(\mathbb T_d)$ in which two vertices are adjacent whenever they lie in a
common animal $A$ with $N_A^\alpha\ge1$, and let
$\mathsf d_{\mathcal Z^\mu_\alpha}$ be its graph distance. Since each animal is connected, this auxiliary graph has the same
connected components as the occupied zoo graph. Put
\[
S_k(\alpha):=\{v:\mathsf d_{\mathcal Z^\mu_\alpha}(\0,v)=k\}.
\]
The shells partition $C^\mu_\alpha(\0)$, and therefore
\[
\chi_\mu(\alpha)=\sum_{k\ge0}\E_\alpha[|S_k(\alpha)|]<\infty.
\]
Choose $K\ge1$ such that
\[
m_K(\alpha):=\E_\alpha[|S_K(\alpha)|]<1.
\]

We first record the one-step exponential moment. Let
\[
\mathcal N_\alpha(x):=\{v\neq x:\exists A\in\mathfrak A
\text{ with }N_A^\alpha\ge1\text{ and }x,v\in V(A)\}.
\]
The variable $|\mathcal N_\alpha(x)|$ is bounded by
\[
\sum_{A:\,x\in V(A)} |V(A)|\,N_A^\alpha.
\]
Hence, by \eqref{eq:zexp}, for all sufficiently small
$\lambda>0$,
\[
\E[e^{\lambda |\mathcal N_\alpha(x)|}]
\le
\exp\left\{\alpha\sum_{A:\,x\in V(A)}\mu(A)
\big(e^{\lambda |V(A)|}-1\big)\right\}
<\infty.
\]

As in the loop soup block argument, explore the zoo-distance ball from a vertex
one generation at a time. If every discovered vertex is instead given an
independent fresh copy of the full one-step neighborhood $\mathcal N_\alpha(x)$, collisions
are removed and only extra vertices are added. Since $K$ is fixed, the
zoo-distance block
\[
B:=|\{v:\mathsf d_{\mathcal Z^\mu_\alpha}(\0,v)<K\}|
\]
and the shell variable
\[
Y:=|S_K(\alpha)|
\]
are dominated by finitely many generations of a BGW process whose
offspring variable has the law of $|\mathcal N_\alpha(\0)|$. Thus $(B,Y)$ has a joint
exponential moment, and $\E[Y]=m_K(\alpha)<1$.

Now run the $K$-block exploration. When an active vertex $x$ is processed,
reveal the still-unexplored animals needed to determine the zoo-distance
vertices at distances $<K$ and $K$ from $x$. Conditional on the past, the
unrevealed animals form a Poisson process on a remaining set of animal classes.
Adding an independent Poisson process on the omitted classes gives an
independent full zoo with the law seen from $\0$ and therefore a stochastic
upper bound for the conditional block. Thus each block is dominated by an
independent copy of $(B,Y)$; keeping duplicate active vertices as distinct
particles only enlarges the exploration.

Every vertex of $C^\mu_\alpha(\0)$ lies within zoo-distance $<K$ of some active
particle along a shortest zoo-distance path from $\0$. Therefore
$|C^\mu_\alpha(\0)|$ is stochastically dominated by the total weight of the
weighted BGW process with mark law $(B,Y)$.
Lemma~\ref{lem:wgwexp} gives the required exponential tail.

Under the hypotheses of Theorem~\ref{thm:zsharp},
$\alpha_{\mathsf c}^\mu=\alpha_{\#}^\mu$, and every
$\alpha<\alpha_{\#}^\mu$ has finite susceptibility by monotonicity and
the definition of $\alpha_{\#}^\mu$.
\end{proof}

\noindent\textbf{Funding acknowledgement.} AK's research is funded by the Cusanuswerk e.V.

\appendix

\section{Exact computation for \texorpdfstring{$S_5$}{S5}} \label{sec:exComp}

Here, we provide the details of the numerical evaluation of the bounds in the proof of Theorem~\ref{thm: loop soup}. The exact closed-walk formula is
\begin{equation}\label{eq:closed-walk}
p_{2m}^{(5)}(0)
=
\frac{(2m)!}{10^{2m}}
\sum_{\substack{k_1+\cdots+k_5=m\\ k_i\ge0}}
\prod_{i=1}^5\frac{1}{(k_i!)^2}.
\end{equation}
Set
\[
a_m^{(1)}:=\frac{1}{(m!)^2},
\qquad
a_m^{(r+1)}:=\sum_{j=0}^m\frac{a_{m-j}^{(r)}}{(j!)^2},
\qquad 1\le r\le4.
\]
Then the inner sum in \eqref{eq:closed-walk} is $a_m^{(5)}$, so
all terms in the following finite computation are rational.

The target bound \eqref{eq:finite-sum-bound} was obtained with the following Python~3 code. The program was written with the help of Codex-GPT 5.5 and verified by the authors. The class
\texttt{Fraction} performs the finite computation in exact rational arithmetic;
floating-point arithmetic is used only to print decimal values and to evaluate
the explicit analytic tail bound.

\footnotesize
\begin{verbatim}
from fractions import Fraction
from math import factorial, pi, sqrt

N = 307

# b[m] = a_m^(1), followed by four exact convolutions.
b = [Fraction(1, factorial(m)**2) for m in range(N + 1)]
a = b[:]
for _ in range(4):
    a = [
        sum(
            (a[m-j] * b[j] for j in range(m + 1)),
            Fraction(0),
        )
        for m in range(N + 1)
    ]

# Exact partial sum through m = 307.
S = Fraction(1)
for m in range(1, N + 1):
    p = Fraction(factorial(2*m), 10**(2*m)) * a[m]
    S += (2*m + 1) * p

assert S < Fraction(188939, 100000)  # S < 1.88939

# Ball--Sterbenz applies for integer m >= 308.  The sum
# over m >= 308 is bounded by the integral from 307.
tail = (
    2 * (5 / (4*pi))**2.5
    * (4 / sqrt(307) + 2 / (3 * 307**1.5))
)
assert tail < 0.0457

# The displayed rounded bounds imply S_5 < 1.936
# by exact rational arithmetic.
assert (
    Fraction(188939, 100000) + Fraction(457, 10000)
    < Fraction(1936, 1000)
)

print("finite sum:", float(S))
print("tail bound:", tail)
print("total:", float(S) + tail)
\end{verbatim}

The output is
\begin{verbatim}
finite sum: 1.8893892815971292
tail bound: 0.045620165807501875
total: 1.935009447404631
\end{verbatim}

\printbibliography

\end{document}